\documentclass{amsart}

\usepackage{hyperref}
\hypersetup{%
pdftitle={},%
pdfauthor={K. S.},%
bookmarks, bookmarksopen,%
colorlinks, linkcolor=blue,%
hyperindex=true%
}
\usepackage{fset}
\usepackage{times} 

\title{Lyndon words and Fibonacci numbers}
\author{Kalle Saari}
\address{Department of Mathematics and Statistics,
University of Winnipeg, 515 Portage Avenue, Winnipeg, MB R3B 2E9, Canada}
\email{kasaar2@gmail.com}

\begin{document}

\maketitle

\begin{abstract}
It is a fundamental property of non-letter Lyndon words that they can be expressed as a concatenation of two shorter Lyndon words.
This leads to a naive lower bound $\nceil{\log_{2}(n)} + 1$ for the number of distinct Lyndon factors that a Lyndon word of length $n$ must have,
but this bound is not optimal. In this paper we show that a much more accurate lower bound is $\nceil{\log_{\phi}(n)} + 1$, where $\phi$ denotes the golden ratio
$(1 + \sqrt{5})/2$. We show that this bound is optimal in that it is attained by the Fibonacci Lyndon words.
We then introduce a mapping $\calL_{\bfx}$ that counts the number of Lyndon factors of length at most $n$ in an infinite word~$\bfx$.
We show that a recurrent infinite word $\bfx$ is aperiodic if and only if $\calL_{\bfx} \geq \calL_{\bff}$, where $\bff$ is the Fibonacci infinite word, with equality if and only if $\bfx$ is in the shift orbit closure of~$\bff$. 
\end{abstract}

\keywords \textbf{Keywords:} Lyndon word, Fibonacci word, Central word, Golden ratio, Sturmian word, Periodicity


\section{Introduction}

Lyndon words are primitive words that are the lexicographically smallest words in their conjugacy classes~\cite{Lothaire1997}. 
Originally defined in the context of free Lie algebras~\cite{CheFoxLyn1958}, Lyndon words have shown to be 
a useful tool for a variety of problems in combinatorics ranging from the construction of de Bruijn sequences~\cite{FreMai1978} to
proving the optimal lower bound for the size of uniform unavoidable sets~\cite{ChaHanPer2004}.
One of the fundamental properties of Lyndon words is their recursive nature: if $w$ is a non-letter Lyndon word, 
then there exist two shorter Lyndon words $u$ and $v$ such that $w = uv$~\cite{CheFoxLyn1958}. This implies that the number of different Lyndon factors of $w$ is bounded below by $\bceil{\log_{2}\nabs{w}} + 1$, but a little experimentation shows that this is hardly optimal. One of the results of this paper, 
Corollary~\ref{ti 17-07-2012:1504}, is that a much better lower bound is $\bceil{\log_{\phi}\nabs{w}} + 1$, where $\phi$ denotes the golden ratio $(1 + \sqrt{5})/2$. 
Here the base of the logarithm is optimal, because the Fibonacci Lyndon words attain the lower bound. 
This follows from Theorem~\ref{ti 10-07-2012:1247}, in which we show that if $w$ is a Lyndon word with $\nabs{w} \geq F_{n}$, where $F_{n}$ is the $n^{\text{th}}$ Fibonacci number, then the number of distinct Lyndon factors in $w$ is at least~$n$ with equality if and only if $w$ equals one of the two Fibonacci Lyndon words of length $F_{n}$, up to renaming letters.

It also makes sense to count the number of Lyndon factors of infinite words, but here we have to use caution: if an infinite word is aperiodic, it will have infinitely many Lyndon factors, as we will show in Corollary~\ref{100620121216}. Thus we define a mapping $\calL_{\bfx} \colon \N \rightarrow \N$ for which $\calL_{\bfx}(n)$ is the number of distinct Lyndon words of length at most~$n$ occurring in a given infinite word~$\bfx$. Of special importance is the Fibonacci infinite word $\bff$.
Our first main result in this setting, Theorem~\ref{020720120123}, is that if $\bfx$ is aperiodic, then $\calL_{\bfx} \geq \calL_{\bff}$.
As Lyndon words are unbordered, this is an improvement of a classic result by Ehrenfeucht and Silberger~\cite{EhrSil1979} stating only that an aperiodic infinite word must have arbitrarily long unbordered factors. 
If we confine our realm to recurrent infinite words, then the above result can be improved as follows.  In Theorem~\ref{060720122147} we show that
a recurrent  infinite word $\bfx$ is aperiodic if and only if $\calL_{\bfx} \geq \calL_{\bff}$ with equality if and only if $\bfx$ is in the shift orbit closure of the Fibonacci word~$\bff$, up to renaming letters.

Fibonacci words are sort of a universal optimality prover in that they possess a wide range of extremal properties, see e.g.~\cite{Cassaigne2008,CroRyt1995,DieHarNow2010,MigResSal1998, GatFlo2012}. The problem of the enumeration of Lyndon factors in automatic and linearly recurrent sequences has recently been studied in~\cite{GocSaaSha2012}.

\section{Preliminaries}

In this section we establish the notation  of this paper and present some preliminary results.
We assume the reader is familiar with the usual terminology of words and 
languages as given in~\cite{AllSha2003} or~\cite{Lothaire2002}.

Let $\calA$ be a finite, nonsingular alphabet totally ordered by $<$; thus every pair of distinct letters $\wa, \wb \in \calA$ satisfy either $\wa < \wb$ or $\wb < \wa$, but not both. We use the same symbol `$<$' to denote the usual order relation among the integers, 
but this should not cause problems as the context always tells which order is meant.
In what follows, we sometimes assume that $\o, \i \in \calA$, sometimes $\wa, \wb \in \calA$, and then their mutual order is implicitly assumed to be their ``natural order,'' so that $\o < \i$ and $\wa < \wb$.

The set of all finite words over $\calA$ is denoted by $\calA^{*}$ and the set of finite words excluding the empty word $\epsilon$ is denoted by $\calA^{+}$.

Let $w = a_{1} a_{2} \cdots a_{n}$ be a nonempty finite word  with $a_{i} \in \calA$ and $n \geq 1$. 
The \emph{length} of $w$ is $\nabs{w} = n$; we denote the cardinality of a set $X$ by $\# X$.
The \emph{reversal} of $w$ is the word $w^{R} = a_{n} a_{n-1} \cdots a_{1}$.
If $w^{R} = w$, then $w$ is a \emph{palindrome}.
The word $w$ has period $p\geq 1$ if $a_{i+p} = a_{i}$ for all $i = 1, 2, \ldots, n - p$. According to this definition, any integer $p\geq n$ is a period of $w$.
If $p\leq n$, then $p$ is a period of $w$ if and only if there exist words $x,y,z \in \calA^{*}$ such that $w = xy = zx$ and $\nabs{y} = \nabs{z} = p$. 
If $w$ has no periods  smaller than $\nabs{w}$, then it is called \emph{unbordered}, otherwise $w$ is \emph{bordered}.  
Suppose that $w = pfs$ with $p,f,s \in \calA^{*}$.
Then $p,f,s$ are called a \emph{prefix}, \emph{factor}, and \emph{suffix} of $w$, respectively.
In addition, $p$ and $s$ are \emph{proper} prefix and suffix if they do not equal~$w$.
We say that a word $z\in \calA^{+}$ is a \emph{periodic extension} of $w$ if $z$ is a prefix of a word in $w^{+}$. We abuse the word ``extension''
here in that we allow an ``extension'' to be a prefix of $w$. The word we get from $w$ by deleting its last letter is denoted by $w^{\flat}$; thus $w^{\flat} = a_{1} a_{2}\cdots a_{n-1}$.  Also if $w = xy$ for some words $x,y$, we denote $x^{-1}w = y$ and $wy^{-1} = x$.
The word $w$ is \emph{primitive} if it cannot be written in the form $w = u^{k}$ for a word $u\in \calA^{+}$ and an integer $k\geq2$.
If $w = uv$, then the word $vu$ is called a \emph{conjugate} of~$w$. The set of all conjugates of $w$ is called the \emph{conjugacy class} of $w$.

\begin{lemma}[Castelli, Mignosi, and Restivo~\cite{CasMigRes1999}]\label{080720121124}
Let $w \in \calA^{+}$ be a word with periods $p,q$.
\begin{enumerate}[(i)]
 \item If  $q < p \leq \nabs{w}$, then the prefix and suffix of $w$ of length $\nabs{w} - q$
have periods $q$ and $p-q$.
\item Let $u$ and $v$ be the prefix and suffix of $w$ of length $q$, respectively. Then $uw$ and $wv$ have periods $q$ and $p+q$. 
\end{enumerate}
\end{lemma}
In the property~(ii) in the previous lemma, the indicated source~\cite{CasMigRes1999} only mentions and proves the claim for the periods of $uw$,
but the case for the periods of $wv$ can be proved similarly.

\begin{lemma}[Fine and Wilf~\cite{FinWil1965}]\label{020720122145}
If a word $w\in\calA^{+}$ has periods $p$ and $q$, and 
\[
p + q - \text{gcd}(p,q) \leq \nabs{w},
\]
then $w$ has period $\text{gcd}(p,q)$. 
\end{lemma}

It is well-known that the above lemma is optimal. That is to say, 
if neither of $p$ and $q$ equals $\gcd(p,q)$, then there exists a word $z\in\calA^{+}$ of length $\nabs{z} = p + q - \text{gcd}(p,q) - 1$ 
that has periods $p$ and $q$, but does not have a period $\text{gcd}(p,q)$.
Following~\cite{Lothaire2002},  we call a word $z \in \{\wa, \wb\}^{*}$ a \emph{central word over $\{\wa, \wb\}$} 
if there exist two coprime integers $p,q$ such that $\nabs{z} = p + q - 2$ and both $p$ and $q$ are periods of~$z$.
Equivalently, $z$ is a central word over $\{\wa, \wb\}$ if either $z \in \wa^{*} \cup \wb^{*}$ or there exist two coprime integers $p,q \geq 2$ such that $\nabs{z} = p + q - 2$, both $p$ and $q$ are periods of $z$, 
but $z$ does not have period $\gcd(p,q) = 1$, that is, both letters $\wa$ and $\wb$ occur in~$z$.
These words are also known as extremal Fine and Wilf words~\cite{TijZam2009}.
They are palindromes and unique up to renaming letters~\cite{Lothaire2002,TijZam2009}.
The latter fact implies that there are exactly two central words over $\{\wa,\wb\}$ with given periods $p$ and $q$; if one is $z$, then the other one 
is~$c(z)$, where $c$ is the morphism $\wa \mapsto \wb$, $\wb \mapsto \wa$.

Recall that the Fibonacci numbers are defined recursively as $F_{0} = 0$,  $F_{1} = 1$, and $F_{n} = F_{n-1} + F_{n-2}$  for $n\geq2$,
and that every two consecutive Fibonacci numbers are coprime. The two central words over $\{\wa,\wb\}$ for periods $F_{n-2}$ and $F_{n-1}$ can be
obtained  by means of finite Fibonacci words as follows. Let $\wa, \wb\in \calA$ be distinct letters. Then define $f_{1} = \wb$,  $f_{2} = \wa$, and
\[
f_{n} = f_{n-1} f_{n-2} \quad (n\geq 3).
\]
We call the words $f_{n}$ \emph{finite Fibonacci words over} $\{ \wa, \wb\}$. Note that we could also have defined $f_{1} = \wa$ and $f_{2} = \wb$.
This makes a difference in our considerations, and we will make this distinction when it matters.
Let $p_{n}$ denote the word for which $f_{n} = p_{n} xy$ with $xy \in \nsset{\wa\wb, \wb\wa}$ and $n\geq3$. 
Thus 
\begin{align*}
f_{3} &= \word{ab} &  f_{4} &= \word{aba} & f_{5} &= \word{abaab}   & f_{6} &= \word{abaabaab} & f_{7} &= \word{abaababaabaab} \\
p_{3} &= \epsilon &  p_{4} &= \word{a} & p_{5} &= \word{aba}  &      p_{6} &= \word{abaaba} &      p_{7} &= \word{abaababaaba}
\end{align*}
Note that $\nabs{f_{n}} = F_{n}$ and thus $\nabs{p_{n}} = F_{n} - 2$. It can be shown that 
if $n\geq 5$, then $p_{n}$ has periods $F_{n-2}$ and $F_{n-1}$, but it does not have period $\gcd(F_{n-2}, F_{n-1}) = 1$.
Thus each $p_{n}$ is a central word over $\{\wa,\wb\}$. The other central word with the same periods is  
$c(p_{n})$, where $c$ is the morphism $\wa \mapsto \wb$ and $\wb \mapsto \wa$.
We will give some further properties of the Fibonacci words in Lemmas~\ref{ma 16-07-2012:1521} and~\ref{ma 16-07-2012:1209}, but let us stress here that a word over $\{\wa, \wb\}$ that is  of length $F_n -2\geq 2$ is one of $p_n$ or $c(p_n)$ if and only if it has periods $F_{n-1}$ and $F_{n-2}$ but does not have period~1.

The order $<$ of $\calA$ is extended to $\calA^{*}$ as follows: For  $u,v \in \calA^{*}$ we have
\[
 u < v \qIff
\begin{cases}
 \text{$u$ is a prefix of $v$,  or} \\
 \text{$u=x\wa u'$ and $v=x\wb v'$ with $x, u', v' \in \calA^{*}$,  $\wa,\wb \in \calA$ and  $\wa < \wb$.} 
\end{cases}
\]
This is called the \emph{lexicographic ordering} of $\calA^{*}$ with respect to $<$.

A nonempty, primitive word $w\in \calA^{+}$  is called a \emph{Lyndon word}  if it is the smallest word in its conjugacy class. In particular, letters are Lyndon words, but the empty word is not. For example, the Lyndon words $w \in \bina^{+}$  with $\nabs{w} \leq 4$ are
\[
\o, \quad \i, \quad \o\i, \quad \o\o\i, \quad \o\i\i, \quad \o\o\o\i, \quad \o\o\i\i, \quad \o\i\i\i.
\]
\begin{lemma}[Berstel and de Luca~\cite{BerdeL1997}]\label{140720121241}
If $z \in\calA^{*}$ is a central word over $\{\wa,\wb\}$ and $\wa < \wb$, then $\wa z \wb$ is a Lyndon word.
\end{lemma}

According to Lemma~\ref{140720121241}, the words $\wa p_{n} \wb $ and $\wa c(p_{n}) \wb$ are Lyndon words; we call them 
\emph{Fibonacci Lyndon words of length~$F_{n}$ over $\{\wa, \wb\}$}. 
The first few ones are
\begin{align*}
\word{ab} && \word{aab} && \word{aabab} && \word{aabaabab} && \word{aabaababaabab} \\
\word{ab} && \word{abb} && \word{ababb} && \word{ababbabb} && \word{ababbababbabb}
\end{align*}
Here the upper row corresponds to words $\word{a}p_{n}\word{b}$ and the lower row to words $\word{a}c(p_{n})\word{b}$.

The properties of Lyndon words given in the next lemma are well-known~\cite{Lothaire1997, Duval1983}.
\begin{lemma}\label{070820102200}
Lyndon words have the following properties.
\begin{enumerate}[(i)]
\item Lyndon words are unbordered.
\item A word $w$ is Lyndon if and only if $w < y$ for all nonempty proper suffixes $y$ of $w$.
\item If $u$ and $v$ are Lyndon words and $u<v$, then $uv$ is a Lyndon word.
\end{enumerate} 
\end{lemma}

\begin{lemma}\label{100620121631}
Let $w\in\calA^{+}$ be a Lyndon word. Suppose that $z\wa$ is a periodic extension of $w$ for some $z \in \calA^{+}$ and $\wa, \wb \in \calA$ with 
$\wa<\wb$. Then $z\wb$ is a Lyndon word.
\end{lemma}
\begin{proof}
There exist an integer $n\geq 0$ and words $x,y \in\calA^{*}$ such that $w = x\wa y$ and $z = w^{n}x$. 
We show that $x\wb$ is a Lyndon word; this suffices because then Lemma~\ref{070820102200} implies that $z\wb = w^{n}x\wb$ is a Lyndon word since $w < x\wb$.

Contrary to what we want to show, suppose that $x\wb$ is not a Lyndon word.  Then Lemma~\ref{070820102200} implies that $x\wb$ has a nonempty suffix $v$
such that $v < x\wb$. Write $v = v' \wb$. Then $v' \wa y$ is a suffix of $w$, so that $v' \wa y > w$ because $w$ is a Lyndon word. Therefore $\wb > \wa$ implies that $v = v'\wb$ is not a prefix of $w$ and thus not a prefix of~$x$.
Consequently $v < x\wb$ implies that we can write $v = t \texttt{c} t'$ and $x\wb = t \texttt{d} t''$ for some words $t, t', t''\in \calA^{*}$ and letters $\texttt{c},\texttt{d}\in \calA$ with $\texttt{c} < \texttt{d}$.
Since $\nabs{t\texttt{d}} = \nabs{t \texttt{c}}  \leq \nabs{v} < \nabs{x\wb}$, the word $t\texttt{d}$ is a prefix of $x$ and thus a prefix of~$w$. But then
\[
v'\wa y < v'\wb y = vy = t\texttt{c}t' y < t\texttt{d} < w,  
\]
a contradiction. Thus $x\wb$ is a Lyndon word, and the proof is complete.
\end{proof}

\begin{lemma}\label{pe 15-06-2012:0726}
Let $w\in \calA^{+}$ be a Lyndon word with $\nabs{w} \geq 2$. Let $\lambda_{w}$ be the longest proper prefix of~$w$ that is also a Lyndon word.
Then $w^{\flat}$ is a periodic extension of $\lambda_{w}$. Furthermore, the word $\mu_{w} := \lambda_{w}^{-1}w$ is a Lyndon word.
\end{lemma}
\begin{proof}
Suppose that $w^{\flat}$ is not a periodic extension of $\lambda_{w}$. 
Then there exist a word $u$ and different letters $\wa, \wb$ such that $u\wa$ is a prefix of $\lambda_{w}$ and $\lambda_{w}^{k}u\wb$ is a prefix of $w^{\flat}$ for some integer~$k\geq 1$.  
If $\wa < \wb$, then  $\lambda_{w}^{k}u\wb$ is a Lyndon word by Lemma~\ref{100620121631}, contradicting the maximality of $\nabs{\lambda_{w}}$.
If $\wa > \wb$, then $(\lambda_{w}^{k})^{-1}w < w$ because $(\lambda_{w}^{k})^{-1}w$ begins with $u\wb$ while $w$ begins with $u\wa$.
Thus $w$ has a nonempty suffix that is smaller than $w$, contradicting Lemma~\ref{070820102200} because $w$ is Lyndon.
Therefore $w^{\flat}$ is a periodic extension of $\lambda_{w}$,
and consequently there exist $u\in\calA^{+}$ and letters $\wa, \wb \in \calA$ such that $u\wa$ is a prefix of $\lambda_{w}$ and $w = \lambda_{w}^{k}u\wb$ for some integer~$k\geq 1$.  Furthermore, we must have $\wa <\wb$ because otherwise $u \wb < w$, which is impossible because 
$w$ is a Lyndon word and $u \wb$ its suffix.
Now $\mu_{w} = \lambda_{w}^{-1} w = \lambda_{w}^{k-1}u\wb$ is a Lyndon word by Lemma~\ref{100620121631}.
\end{proof}

Due to its importance in the upcoming considerations, let us restate Lemma~\ref{pe 15-06-2012:0726}:  every non-letter Lyndon word $w\in \calA^{+}$ can be written as $w 
= \lambda_{w}\mu_{w}$, where $\lambda_{w}$ and $\mu_{w}$ are Lyndon words and $w^{\flat}$ is a periodic extension of $\lambda_{w}$.

An \emph{infinite word} is a sequence $\bfx = a_{1} a_{2} a_{3} \ldots a_{n} \ldots$ where $a_{n} \in \calA$.
The set of infinite words over $\calA$ is denoted by $\calA^{\N}$.
A \emph{tail} of the infinite word $\bfx$ is another infinite word $\bfy \in \calA^{\N}$ such that $\bfx = x \bfy$ for some $x\in \calA^{*}$.
An infinite word \bfx is \emph{purely periodic} if $\bfx = uuu \ldots u \ldots$ for some finite word $u \in \calA^{+}$; we also denote this by $\bfx = u^{\omega}$.
The word $\bfx$ is \emph{ultimately periodic} if it has a purely periodic tail.
Finally, $\bfx$ is \emph{aperiodic} if it is not ultimately periodic.
A \emph{factor} of $\bfx$ is a finite word that occurs somewhere in $\bfx$.
The set  of all factors of $\bfx$ is denoted by $F(\bfx)$. The word $\bfx$ is called \emph{recurrent} if each of its factors occurs at least twice (and thus infinitely many times) in $\bfx$.
If an infinite word is ultimately periodic and recurrent, then it can be shown to be purely periodic.
The \emph{shift orbit closure} of $\bfx$ is the set of all infinite words $\bfy\in \calA^{\N}$ such that $F(\bfy) \subseteq F(\bfx)$.

Let $f_{n}$ be the finite Fibonacci words over $\{\wa, \wb\}$ defined above with $f_{1} = \wb$ and $f_{2} = \wa$.
If $n\geq2$, then $f_{n}$ is a prefix of $f_{n+1}$. Thus there exists a unique infinite word $\bff$ such that $f_{n}$ is a prefix of $\bff$ for every $n\geq 2$. The word $\bff$ is called a \emph{Fibonacci infinite word} over $\{\wa, \wb\}$. Note that there is another Fibonacci infinite word over $\{ \wa, \wb\}$, which results from defining $f_{1} = \wa$ and $f_{2} = \wb$. When we want to stress the definition of $f_{1}$ and $f_{2}$ when constructing $\bff$, we denote $\bff = \lim_{\ntoinf}f_{n}$. The Fibonacci infinite words are easily seen to be recurrent,
and it can be shown that they are even \emph{uniformly recurrent}, which means that if $w$ is a factor of $\bff$, then there exists an integer $N\geq 1$ such that every factor of $\bff$ of length $N \nabs{w}$ has a factor~$w$. This implies that  $\bfx$ is in the shift orbit closure of $\bff$ if and only if $F(\bfx) = F(\bff)$.


\begin{lemma}\label{ma 16-07-2012:1521}
Let $f_{n}$ denote the Fibonacci words, $p_{n}$ the corresponding central words, and $\bff = \lim_{\ntoinf} f_{n}$.
\begin{enumerate}
 \item The words $p_{n}$ are palindromes~\cite{deLuca1981}.
 \item The Lyndon factors of $\bff$ are precisely the Lyndon conjugates of $f_{n}$ \cite[Lemma~7]{CurSaa2009}.
 \item If $w$ is a conjugate of $f_{n}$, then its reversal $w^{R}$ is  a conjugate of $f_{n}$~\cite{WenWen1994}.
\end{enumerate}
 
\end{lemma}

\begin{lemma}\label{ma 16-07-2012:1209}
Let $\bff = \lim_{\ntoinf} f_{n}$ be a Fibonacci word for which $\nsset{f_{1}, f_{2}} = \{\wa, \wb\}$. Write $f_{n} = p_{n}xy$ with $xy \in \{\wa\wb, \wb\wa\}$  and suppose that $\wa < \wb$. Then the Lyndon conjugate of $f_{n}$ is the word $\wa p_{n} \wb$ for all $n\geq3$.
Furthermore, every Lyndon factor of $\bff$ that is shorter than $\wa p_{n} \wb$ is either a prefix or a suffix of $\wa p_{n} \wb$.
\end{lemma}
\begin{proof}
First off, the word $\wa p_{n} \wb$ is a Lyndon word by Lemma~\ref{140720121241}. Thus the first claim is proved by showing that $\wa p_{n} \wb$ is a conjugate of $f_{n}$.
This is clear if $f_{n} = p_{n} \wb \wa$, so assume that $f_{n} = p_{n} \wa \wb$ instead.
Then $\wb p_{n} \wa$ is a conjugate of $f_{n}$. Since the reversal of a conjugate of $f_{n}$ is a conjugate of $f_{n}$ and since $p_{n}$ is a palindrome by Lemma~\ref{ma 16-07-2012:1521}, it follows that $\wa p_{n}^{R} \wb = \wa p_{n} \wb$ is a conjugate of $f_{n}$.

Next we show that if $k <  n$, then the Lyndon conjugate of $f_{k}$ is a prefix or a suffix of $\wa p_{n} \wb$. 
Since $\nsset{f_{1}, f_{2}} = \nsset{\wa, \wb}$, the claim is plainly true for $k=1,2$. 
Furthermore, if $f_{2} = \wa$, then the Lyndon conjugate of $f_{3}$, which is $\wa\wb$, is a suffix of $\wa p_{n} \wb$; and 
if $f_{2} = \wb$, then the Lyndon conjugate of $f_{3}$ is a prefix of $\wa p_{n} \wb$.
Thus we may suppose that $k\geq 4$. Then $k < n$ implies that $f_{k}$ is a prefix of~$p_{n}$. Furthermore, since $p_{n}$ is a palindrome, the reversal $f_{k}^{R}$ is a suffix of $p_{n}$.
Therefore if $f_{k} = p_{k} \wb \wa$, then its Lyndon conjugate $\wa p_{k} \wb$ is a prefix of $\wa p_{n} \wb$; and if $f_{k} = p_{k} \wa \wb$, then
its Lyndon conjugate $\wa p_{k} \wb$ is a suffix of $\wa p_{n}\wb$.
\end{proof}

\section{Lyndon factors of Lyndon words}\label{to 15-11-2012 14:27 winnipeg}

Let $w \in \calA^{+}$ be a Lyndon word.
We denote the number of distinct Lyndon factors of $w$ by~$\calL(w)$. 
A trivial but useful observation is that if $\nabs{w}\geq 2$, then 
\[
\calL(w) \geq \calL(\lambda_{w}) + 1 \qqtext{and} \calL(w) \geq \calL(\mu_{w}) + 1,
\]
where $\lambda_{w}$ and $\mu_{w}$ are the Lyndon words provided by Lemma~\ref{pe 15-06-2012:0726}.
If $\nabs{w} \geq 2$, let  $p_{w}$ denote  the word such that $w = \wa p_{w} \wb$ for some letters $\wa,\wb \in \calA$.

\begin{lemma}\label{020720122029}
If $w \in \calA^{+}$ is a Fibonacci Lyndon word of length $F_{n}$ with $n\geq 3$, then $\calL(w) = n$.
\end{lemma}
\begin{proof}
The word $w$ is the Lyndon conjugate of~$f_{n}$ (for some choice of $f_{1}, f_{2}\in \calA$). 
Each of the Lyndon conjugates of $f_{k}$ with $1 \leq k \leq n$ are either prefixes or suffixes of $w$ by Lemma~\ref{ma 16-07-2012:1209}, and these are the only Lyndon factors of $w$ by Lemma~\ref{ma 16-07-2012:1521}. Thus $\calL(w) = n$.
\end{proof}

\begin{lemma}\label{la140720121644}
Let $w \in \calA^{+}$ be a Lyndon word with $\nabs{w} \geq F_{n}$ for $n\geq3$ and let $\wa, \wb \in \calA$ be the letters such that $w = \wa p_{w} \wb$.
Then $w$ is a Fibonacci Lyndon word over $\{ \wa, \wb \}$ of length $F_{n}$ if and only if $\wa p_{w}$ has period $F_{n-1}$ and $p_{w} \wb$ has period $F_{n-2}$, or vice versa.
\end{lemma}
\begin{proof}
The claim is readily verified for $n \leq 4$, so assume that $n>4$.

Suppose first that $w$ is a Fibonacci Lyndon word over $\{\wa, \wb\}$ of length $F_{n}$. 
Since $n \geq 5$, the word $p_{w} = p_{n}$  does not have period $\gcd(F_{n-2}, F_{n-1}) = 1$, but it has periods $F_{n-2}$ and $F_{n-1}$.
We claim that the word $\wa p_{w}$ has period $F_{n-2}$ or $F_{n-1}$. Indeed,  if $p_{w}$ did not have either period, then $\wb p_{w}$ would have both periods, contradicting Lemma~\ref{020720122145} because $\nabs{p_{w}} = F_{n} - 2$ and $\wb p_{w}$ does not have period~1. An analogous argument shows that $p_{w} \wb$ must have period either $F_{n-2}$ or $F_{n-1}$.
Finally, $\wa p_{w}$ and $p_{w}\wb$ cannot have the same period $F_{n-2}$ or $F_{n-1}$ because otherwise $w$ would have the same period, which it does not since it is unbordered by Lemma~\ref{070820102200}.

Conversely, suppose that $\wa p_{w}$ has period $F_{n-1}$ and $p_{w} \wb$ has period $F_{n-2}$, or vice versa.
Then since $n\geq 5$, we have $F_{n-1} < \nabs{\wa p_{w}} = \nabs{p_{w}\wb}$. This implies that both $\wa$ and $\wb$ occur in $p_{w}$, and thus $\gcd(F_{n-1}, F_{n-2}) = 1$ is not a period of~$p_{w}$. Since $p_{w}$ does have periods $F_{n-2}$ and $F_{n-1}$ and length $\geq F_{n} - 2$,
it follows that actually $\nabs{p_{w}} = F_{n} - 2$ and that $p_{w}$ is a central word. Thus $w$ is a Fibonacci Lyndon word over $\{\wa, \wb\}$ of length~$F_{n}$.
\end{proof}

\begin{lemma}\label{160620121145}
Let $w \in \calA^{+}$ be a Lyndon word with $\nabs{w} \geq F_{n}$ for some $n\geq 3$.
Then we have  $\calL(w) \geq n$.
Furthermore if $\calL(w) = n$, then $w$ is a Fibonacci Lyndon word of length~$F_{n}$.
\end{lemma}
\begin{proof}
Using the fact that letters are Lyndon words and that $w$ is a product of two shorter Lyndon words, the reader readily verifies the claim for 
$n\leq 4$.
Hence we assume inductively that $n\geq 5$ and the claim holds for all values $<n$. 
Since $w$ is a Lyndon word, we have $w = \lambda_{w}\mu_{w} = \wa p_{w} \wb$ for some letters $\wa, \wb \in \calA$ with $\wa < \wb$.
We split the proof in several cases depending on the length of~$\lambda_{w}$.

\emph{Case (i)}. Suppose that $\nabs{\lambda_{w}} > F_{n-1} $. Since $n-1\geq 4$,  we may apply the induction assumption to $\lambda_{w}$,
which gives that $\calL(\lambda_{w}) > n - 1$, and so
\[
\calL(w) \geq \calL(\lambda_{w}) + 1 \geq n + 1 > n.
\]

\emph{Case (ii)}. Suppose that $\nabs{\lambda_{w}} = F_{n-1}$. Then $\nabs{\mu_{w}} \geq F_{n-2}$, and we have three subcases:

\emph{Case (ii-a)}. If $\mu_{w}$ is not a factor of $\lambda_{w}$, then the induction assumption implies
 \[
 \calL(w) \geq \calL(\lambda_{w}) + \# \nsset{w, \mu_{w}} \geq (n- 1)  + 2 > n.
 \]
 
\emph{Case (ii-b)}. Suppose that $\mu_{w}$ is a factor but not a suffix of $\lambda_{w}$. Then by denoting the longest Lyndon prefix  of~$\lambda_{w}$ 
by~$\lambda_{\lambda_{w}}$, we have $\nabs{\lambda_{\lambda_{w}}} \geq \nabs{\mu_{w}}$ because $\lambda_{w}^{\flat}$ is a periodic extension of $\lambda_{\lambda_{w}}$ by Lemma~\ref{pe 15-06-2012:0726} and $\mu_{w}$ is unbordered by Lemma~\ref{070820102200}. Therefore $\nabs{\lambda_{\lambda_{w}}} \geq F_{n-2}$. Since $n-2\geq 3$, we may apply the induction assumption to $\lambda_{\lambda_{w}}$, obtaining
\begin{equation}\label{150620122237}
\calL(w) \geq \calL(\lambda_{w}) + 1 \geq \calL(\lambda_{\lambda_{w}}) + 2  \geq (n-2) + 2 = n,
\end{equation}
where the last inequality is equality only if $\nabs{\mu_{w}} = \nabs{\lambda_{\lambda_{w}}} = F_{n-2}$.
This would imply that $\lambda_{\lambda_{w}}$ and $\mu_{w}$ are conjugates, which would further imply that  $\lambda_{\lambda_{w}} = \mu_{w}$ because both are Lyndon words. But then $\mu_{w}$ is both a prefix and a suffix of $w$, contradicting the fact that $w$ is unbordered. Hence the third inequality in~\eqref{150620122237} is strict.
 
\emph{Case (ii-c)}. Suppose that $\mu_{w}$ is a suffix of $\lambda_{w}$. 
Then it is  a suffix of~$p_{\lambda_{w}}\wb$ because~$\mu_{w}$ cannot equal $\lambda_{w}$.
Since $\nabs{\lambda_{w}} =  F_{n-1}$, the induction assumption gives
\[
 \calL(w) \geq \calL(\lambda_{w}) + 1 \geq (n - 1) + 1 = n.
\]
If $\calL(w) > n$, we are done, so assume $\calL(w) = n$; then  $\calL(\lambda_{w}) = n - 1$. 
We will show that $\wa p_{w}$ has period $F_{n-1}$ and $p_{w}\wb$ period $F_{n-2}$, which means that $w$ is a Fibonacci Lyndon word by Lemma~\ref{la140720121644}.

First, the word $\wa p_{w}$ has period $\nabs{\lambda_{w}} = F_{n-1}$ because $\wa p_{w} = w^{\flat}$ is a periodic extension of $\lambda_{w}$ by 
Lemma~\ref{pe 15-06-2012:0726}.
Second, since $\nabs{\lambda_{w}} = F_{n-1}$ and $\calL(\lambda_{w}) = n - 1$, the induction assumption implies that $\lambda_{w}$ is a Fibonacci Lyndon word
and therefore Lemma~\ref{pe 15-06-2012:0726} implies that $p_{\lambda_{w}}\wb$ has period either $F_{n-3}$ or $F_{n-2}$.
The period cannot be $F_{n-3}$, however, because the unbordered word $\mu_{w}$ is a suffix of $p_{\lambda_{w}}\wb$ and $\nabs{\mu_{w}} \geq F_{n-2} > F_{n-3}$.
Thus $p_{\lambda_{w}}\wb$  has period $F_{n-2}$, and furthermore $\nabs{\mu_{w}} = F_{n-2}$. Now the fact that $\mu_{w}$ is a suffix of $p_{\lambda_{w}}\wb$ with $\nabs{\mu_{w}} = F_{n-2}$ and that $p_{\lambda_{w}}\wb$ has period $F_{n-2}$ 
imply that $p_{w}\wb = p_{\lambda_{w}}\wb \mu_{w}$ has period $F_{n-2}$ by Lemma~\ref{080720121124}.

\emph{Case (iii)}. Suppose that $\nabs{w}/2 < \nabs{\lambda_{w}} < F_{n-1}$. Then $\nabs{\mu_{w}} = \nabs{w} - \nabs{\lambda_{w}}$
gives $F_{n-2} < \nabs{\mu_{w}} < \nabs{\lambda_{w}}$, so $\lambda_{w}$ is not a factor of $\mu_{w}$. Furthermore, since $n-2 \geq 3$, 
the induction assumption gives
\[
\calL(w) \geq \calL(\mu_{w}) + \# \nsset{w, \lambda_{w}} > (n-2) + 2 = n.
\]

\emph{Case (iv)}. Suppose that $\nabs{\lambda_{w}} = \nabs{w}/2$.  Then $\nabs{\mu_{w}} = \nabs{\lambda_{w}}$ and thus $\mu_{w}$ is not a factor of~$\lambda_{w}$ because otherwise $\mu_{w} = \lambda_{w}$ and $w$ would be bordered. Noting that $\nabs{\lambda_{w}} = \nabs{w}/2 > F_{n-2}$ because $n\geq4$, we therefore have
by induction
\[
\calL(w) \geq \calL(\lambda_{w}) + \# \nsset{w, \mu_{w}} > (n - 2) + 2 = n.
\]

\emph{Case (v)}. Suppose that $F_{n-2} < \nabs{\lambda_{w}} < \nabs{w}/2$. Then $\nabs{ \lambda_{w} } < \nabs{\mu_{w}}$,
and so $\mu_{w}$ is not a factor of $\lambda_{w}$. Since we also have $n - 2\geq 3$, the induction assumption gives
\[
\calL(w) \geq \calL(\lambda_{w}) + \# \nsset{w, \mu_{w}} > (n - 2) + 2 = n.
\]

\emph{Case (vi)}. Suppose that $\nabs{\lambda_{w}} = F_{n-2}$.
Then $\nabs{\mu_{w}} \geq F_{n-1}$, so the induction assumption gives
\begin{equation}\label{150620122235}
 \calL(w) \geq \calL(\mu_{w}) + 1 \geq (n - 1) + 1 = n.
\end{equation}
If $\calL(w) > n$, we are done, so assume $\calL(w) = n$. Our goal is to show that $\wa p_{w}$ has period $F_{n-2}$ and $p_{w}\wb$ has period $F_{n-1}$,
which means that $w$ is a Fibonacci Lyndon word by Lemma~\ref{la140720121644}.
The first objective is easy because $w^{\flat}=\wa p_{w}$ is a periodic extension of $\lambda_{w}$ by Lemma~\ref{pe 15-06-2012:0726}, and thus $\wa p_{w}$ has period $\nabs{\lambda_{w}} = F_{n-2}$. 

For the second objective, we take care of a special case first. If $n=5$, then $\nabs{\lambda_{w}} = 2$, so that $\lambda_{w} = \wa\wb$.
Then $w = \wa\wb\wa\wb\wb$ because  $w^{\flat}$ is a periodic extension of $\lambda_{w}$ by Lemma~\ref{pe 15-06-2012:0726}.
Consequently $p_{w} \wb = \word{babb}$ has period $F_{4}$, as claimed.
We may thus assume that $n\geq 6$.

Now, note that since $\calL(w) = n$ in Eq.~\eqref{150620122235}, we have $\calL(\mu_{w}) = n-1$. Since also $\nabs{\mu_{w}} \geq F_{n-1}$, 
the induction assumption says that actually $\nabs{\mu_{w}} = F_{n-1}$ and that 
$\mu_{w}$ is a Fibonacci Lyndon word. Thus Lemma~\ref{pe 15-06-2012:0726} implies that one of $\wa p_{\mu_{w}}$ and $p_{\mu_{w}}\wb$ has period $F_{n-3}$  and the other one has period $F_{n-2}$.
Since $ w ^{\flat} = \lambda_{w} \wa p_{\mu_{w}}$ is a periodic extension of $\lambda_{w}$ and $\nabs{\lambda_{w}} < \nabs{\mu_{w}}$, we see that $\lambda_{w}$ is a prefix of $\wa p_{\mu_{w}}$. Therefore $\wa p_{\mu_{w}}$ cannot have period $F_{n-3}$ because $F_{n-3} < \nabs{\lambda_{w}}$ and $\lambda_{w}$ is unbordered.
Hence $\wa p_{\mu_{w}}$ has period $F_{n-2}$ and $p_{\mu_{w}}\wb$ has period~$F_{n-3}$.
Since $n \geq 6$, we have $\nabs{p_{\lambda_{w}} \wb\wa} = F_{n-2} \leq F_{n-1} - 2 = \nabs{p_{\mu_{w}}}$, and consequently since 
$p_{w} = p_{\lambda_{w}}\wb\wa p_{\mu_{w}}$ and  $p_{w}$ has period $F_{n-2}$ (because $\wa p_{w}$ has period $F_{n-2}$), it follows that $p_{\lambda_{w}} \wb\wa$ is a prefix of $p_{\mu_{w}}$.  
Since $p_{\mu_{w}}$ has periods $F_{n-3}$ and $F_{n-2}$, Lemma~\ref{080720121124}  implies that 
$p_{w}$ has periods $F_{n-2}$ and $F_{n-2} + F_{n-3} = F_{n-1}$. Now, as we have reasoned before, $p_{w}\wb$ must have period either $F_{n-2}$ or $F_{n-1}$ for otherwise $p_{w}\wa$ would have both periods contradicting Lemma~\ref{020720122145}. Since $\wa p_{w}$ has period $F_{n-2}$ and $w$ is unbordered, we conclude that $p_{w}\wb$ must have period $F_{n-1}$.

\emph{Case (vii)}. Suppose that $\nabs{\lambda_{w}} < F_{n-2}$. Then $\nabs{\mu_{w}} > F_{n-1}$, so that the induction assumption implies
\[
\calL(w) \geq \calL(\mu_{w}) + 1 > (n-1)  + 1 = n.
\]

\end{proof}

\begin{theorem}\label{ti 10-07-2012:1247}
Let $w \in \calA^{+}$ be a Lyndon word with $\nabs{w} \geq F_{n}$ for some $n\geq 3$.
Then $\calL(w) \geq n$ with equality if and only if $w$ is a Fibonacci Lyndon word of length~$F_{n}$.
\end{theorem}
\begin{proof}
The claim is obtained by combining Lemmas~\ref{020720122029} and~\ref{160620121145}.
\end{proof}


Recall that $\phi$ denotes the golden ratio $(1 + \sqrt{5})/2$.
\begin{corollary}\label{ti 17-07-2012:1504}
 If $w \in \calA^{+}$ is a Lyndon word, then $\calL(w) \geq \bceil{\log_{\phi}\nabs{w}} + 1$ with equality if $w$ is a Fibonacci Lyndon word.
\end{corollary}
\begin{proof}
The claim is trivial if $\nabs{w} = 1$, so suppose that $\nabs{w} \geq 2$, and 
let $n\geq3$ be the unique integer for which $F_{n} \leq \nabs{w} < F_{n+1}$.

It is well-known~\cite{Knuth} that $\phi^{m-1} < F_{m+1} < \phi^{m}$ for all $m\geq 2$, and this implies $\bceil{\log_{\phi}F_{m+1}} = m$.
Therefore, if $w$ is a Fibonacci Lyndon word, then $\nabs{w} = F_{n}$ and $\calL(w) = n = \bceil{\log_{\phi}F_{n}} + 1$ by Theorem~\ref{ti 10-07-2012:1247}.
If $w$ is not a Fibonacci Lyndon word, then on the one hand, we have $\calL(w) \geq  n+1$ by Theorem~\ref{ti 10-07-2012:1247}.
On the other hand, we have 
\[
\bceil{\log_{\phi}\nabs{w}} \leq  \bceil{\log_{\phi}F_{n+1}}= n.
\]
Combining these two  inequalities gives $\calL(w) \geq \bceil{\log_{\phi}\nabs{w}} + 1$.
\end{proof}

\begin{remark}
A noteworthy feature of Theorem~\ref{ti 10-07-2012:1247} is that the optimal words, the Fibonacci Lyndon words, are made of just two different letters.
A priori it may seem ``obvious'' that this should always be the case for a Lyndon word having the smallest possible number of Lyndon factors, 
but this, in fact, is not true.
For example, each Lyndon word of length 6 has at least 7 Lyndon factors. In this case the Lyndon words with the smallest number of Lyndon factors are, up to renaming the letters, 
\[
\word{000001}\quad\word{000101}\quad\word{001101}\quad\word{010111}\quad\word{010102}\quad\word{010202}\quad\word{021022}\quad\word{011111},
\]
three of which are made of three different letters. However, see Conjecture~\ref{020720121320}.
\end{remark}

\begin{conjecture}\label{020720121320}
Let us denote
\[
\ell(n) = \min \bset{\calL(w)}{ \text{$w$ is a Lyndon word with $\nabs{w} = n$}}
\]
for all $n\geq1$. We conjecture that if $w$ is a Lyndon word with $\nabs{w} \neq 6$ and $\calL(w)= \ell(\nabs{w})$, then $w$ is a \emph{Sturmian Lyndon word}, i.e., 
we have $w \in \{\wa, \wb\}^{+}$, $w = \wa p_{w} \wb$, and $p_{w}$ is a central word.
\end{conjecture}

\section{Lyndon factors of recurrent words}\label{ti 17-07-2012 21:23}

Our main goal in this section is to prove Theorem~\ref{060720122147}, which gives a characterization for aperiodicity of recurrent words by means of its Lyndon factors
using the Fibonacci infinite word. We begin with a few results that are interesting in their own right.

\begin{lemma}[Siromoney, Mathew, Dare, and Subramanian~\cite{SirMatDarSub1994}]\label{060720121742}
Every infinite word $\bfx \in\calA^{\N}$ admits a unique factorization of the form  either 
\[
\bfx = \prod_{i\geq 1} w_{i} \qqtext{or}  \bfx = w_{1} w_{2} \cdots w_{n} \bfx',
\]
where each $w_{i}\in\calA^{+}$ is a finite Lyndon word with $w_{i} \geq w_{i+1}$ and $\bfx' \in\calA^{\N}$ begins with arbitrarily long Lyndon words.
\end{lemma}

\begin{lemma}\label{070720121709}
Let $\bfx \in \calA^{\N}$ be an infinite word. Then at least one of the following holds:
\begin{enumerate}[(i)]
\item $\bfx$ is ultimately periodic;
\item $\bfx$ has a tail that begins with arbitrarily long Lyndon words;
\item $\bfx = \prod_{i\geq 1}w_{i}$, where each $w_{i} \in \calA^{+}$ is a Lyndon word and, for every $k\geq1$, there exists an index $i_{k}$ such that $\nabs{w_{i}} > k$ for all $ i > i_{k}$.
\end{enumerate}
\end{lemma}
\begin{proof}
Let us suppose that $\bfx$ is aperiodic and that it does not have a tail beginning with arbitrarily long Lyndon words;
we will show that then $\bfx$ satisfies property (iii). Let $k\geq1$. Lemma~\ref{060720121742} implies that $\bfx$ admits a factorization $\bfx = \prod_{i\geq 1} w_{i}$ in which the $w_{i} \in \calA^{+}$ are Lyndon words and $w_{i} \geq w_{i+1}$. 
Denote $w = \min \{ w_{i} \, \colon \,  \nabs{w_{i}} \leq k \}$; this word exists because our alphabet $\calA$ is finite. 
Observe that if $w_{i} > w_{i+1}$, then $w_{i} \neq w_{j}$ for all $j > i$.
Furthermore, since $\bfx$ is aperiodic, the sequence of words $w_{i}$ is not ultimately constant.
Consequently, there exists an index
\[
i_{k} = \max\bigl\{j \, \colon \,   w_{j} = w \bigr\}.
\]
Then $\nabs{w_{i}} > k$ whenever $i > i_{k}$. Indeed, if $i > i_{k}$, then the inequality $w_{i_{k}} \geq w_{i}$ and the maximality of $i_{k}$ imply $w_{i_{k}} > w_{i}$. Because of the minimality of $w_{i_{k}}$, we thus have $\nabs{w_{i}} > k$.
\end{proof}

\begin{corollary}\label{100620121216}
If an infinite word $\bfx \in\calA^{\N}$ has only finitely many distinct Lyndon factors, then it is ultimately periodic.
\end{corollary}
\begin{proof}
If $\bfx$ satisfies property (ii) or (iii) in Lemma~\ref{070720121709}, then it clearly has infinitely many Lyndon factors.
Thus $\bfx$ satisfies property (i) and is ultimately periodic.
\end{proof}

\begin{remark}
The sleek proof of Corollary~\ref{100620121216} was suggested by Tero Harju. The author's original proof was more intricate.
\end{remark}

Recall that we denote the set of factors of an infinite word $\bfx$ by $F(\bfx)$.
In what follows, we also denote the set of Lyndon factors of $\bfx$ by $L(\bfx)$. 
(Don't confuse this with the symbol $\calL(w)$ defined in Section~\ref{to 15-11-2012 14:27 winnipeg}.)

\begin{theorem}\label{060720122059}
If $\bfx$ and $\bfy$ are recurrent infinite words and $L(\bfx) = L(\bfy)$, then $F(\bfx) = F(\bfy)$.
\end{theorem}
\begin{proof}
Suppose first that $\bfx$ is ultimately periodic. Then, in fact, it is purely periodic because it is recurrent.
Writing $\bfx = u^{\omega}$, where $u$ is a primitive word, it follows that $\bfx$ has only one Lyndon factor of length 
$\nabs{u}$ ---the Lyndon conjugate of $u$--- and none longer than $\nabs{u}$. Thus $\bfy$ has only finitely many Lyndon factors, so it is ultimately periodic by Corollary~\ref{100620121216}, and hence purely periodic because it is recurrent. Write $\bfy = v^{\omega}$ with $v$ primitive. Since the Lyndon conjugate of $u$ is a factor of $\bfy$ and the Lyndon conjugate of $v$ is a factor of $\bfx$, it follows that $v$ and $u$ are conjugates, and thus $F(\bfx) = F(\bfy)$.

Suppose then that $\bfx$ is aperiodic. We show that every $u \in F(\bfx)$ is a factor of a Lyndon factor of $\bfx$. Since $\bfy$ must be aperiodic as well,
the analogous property clearly holds for~$\bfy$, implying that $F(\bfx) = F(\bfy)$. If $\bfx$ satisfies property~(ii) of Lemma~\ref{070720121709}, then some tail $\bfx'$ of $\bfx$ begins with arbitrarily long Lyndon words. Since $\bfx$ is recurrent, it follows that $u$ is a factor of a Lyndon prefix of~$\bfx'$. 
Thus suppose that $\bfx$ satisfies property~(iii) of Lemma~\ref{070720121709}.
Since $\bfx$ is recurrent, there exists a word $v$ such that $uvu \in F(\bfx)$. Let $k = \nabs{uvu}$ and let $i_{k}$ be the index provided by Lemma~\ref{070720121709}. Since $\bfx$ is recurrent, the word $uvu$ occurs in $w_{i_{k}}w_{i_{k} + 1} w_{i_{k} + 2}\cdots $. Since $\nabs{w_{i_{k}+ j}} > k$ for each $j\geq 1$, it follows that $u$ necessarily occurs in some  $w_{i_{k} + j}$.
\end{proof}

For an infinite word $\bfx \in \calA^{\N}$, we define a mapping $\calL_{\bfx} \colon \N \rightarrow \N$ such that $\calL_{\bfx}(n)$ is the number of Lyndon factors of $\bfx$ of length at most $n$. Notice that the mapping $\calL_{\bfx}$ is increasing, but not necessarily strictly increasing, as can be seen from Lemma~\ref{su 15-07-2012:2022}. Notice also that $\calL_{\bfx}$ determines the number of Lyndon factors of any length $k\geq 1$. Indeed, it it given by the expression
$\calL_{\bfx}(k) - \calL_{\bfx}(k-1)$ for $k\geq 2$.

\begin{lemma}\label{su 15-07-2012:2022}
Let $\bff \in \calA^{\N}$ be a Fibonacci infinite word.
Then for all $n\geq1$, we have $\calL_{\bff}(n) = k$, where $k\geq 2$ is the unique integer such that $F_{k} \leq n < F_{k+1}$.
\end{lemma}
\begin{proof}
Lemma~\ref{ma 16-07-2012:1209} implies that the Lyndon factors in $\bff$ are precisely the Lyndon conjugates of the finite Fibonacci words $f_{k}$.
Therefore if $F_{k} \leq n < F_{k+1}$, the Lyndon factors of length at most $n$ are the Lyndon conjugates of $f_{1}$, $f_{2}$, \ldots, $f_{k}$,
so that $\calL_{\bff}(n) = k$.
\end{proof}

\begin{theorem}\label{020720120123}
If $\bfx \in \calA^{\N}$ is aperiodic, then $\calL_{\bfx} \geq \calL_{\bff}$.
\end{theorem}
\begin{proof}
Since $\calL_{\bfx}$ is increasing, Lemma~\ref{su 15-07-2012:2022} implies that it suffices to show that $\calL_{\bfx}(F_{k}) \geq k$ for all $k\geq 2$.
This is clear for $k=2$ because $\bfx$ is aperiodic.
Thus assume $k\geq3$, and let~$w$ be a shortest Lyndon factor of \bfx  of length $> F_{k}$; Corollary~\ref{100620121216} ensures that such a word~$w$ exists because $\bfx$ is aperiodic. Theorem~\ref{ti 10-07-2012:1247} implies that $\calL(w) > k$. Since $w$ is as short as possible, all of its proper Lyndon factors are of length at most $F_{k}$.
Therefore
\[
\calL_{\bfx}(F_{k}) \geq \calL(w) - 1 \geq k.
\]
\end{proof}

\begin{remark}
 A classic result by Ehrenfeucht and Silberger~\cite{EhrSil1979} states that if an infinite word has only finitely many unbordered factors, then it is ultimately periodic. Since Lyndon words are unbordered, Theorem~\ref{020720120123} is a quantitative formulation of this with an exact lower bound for the necessary number of unbordered factors.
\end{remark}

\begin{remark}
Looking at Theorem~\ref{020720120123}, one might be tempted to postulate that if $\bfx$ is aperiodic, then 
for all $n\geq 1$, the number of length-$n$ Lyndon factors of $\bfx$ must be at least as large as the number of length-$n$ Lyndon factors of a Fibonacci infinite word,
but this is not true. For example, let $\bff$ is the Fibonacci infinite word over $\{\wa, \wb\}$ with $f_{1} = \wb$ and $f_{2} = \wa$,
and let $\bfx = g(\bff)$, where $g$ is the morphism $\wa \mapsto \word{aab}$, $\wb \mapsto \word{aaab}$. Then it is easy to see that $\bfx$ does not have any Lyndon factors of length $5$, while $\bff$ has a Lyndon factor $\wa\wa\wb\wa\wb$.
\end{remark}

\begin{theorem}\label{060720122147}
Let $\bfx \in \calA^{\N}$ be recurrent. Then $\bfx$ is aperiodic if and only if $\calL_{\bfx} \geq \calL_{\bff}$ with equality if and only if
$\bfx$ is in the shift orbit closure of a Fibonacci infinite word~$\bff$.
\end{theorem}
\begin{proof}
We start by proving the first equivalence. Suppose $\bfx$ is ultimately periodic. Then it is purely periodic because it is recurrent. Thus $\calL_{\bfx}$ is ultimately constant, so $\calL_{\bfx}(n)<\calL_{\bff}(n)$ for all sufficiently large~$n$. 
Conversely, if $\calL_{\bfx}(n) < \calL_{\bff}(n)$ for some $n\geq 1$, then Theorem~\ref{020720120123} implies  that $\bfx$ is ultimately periodic.

Let us next prove the second equivalence. 
If $\bfx$ is in the shift orbit closure of $\bff$, then $F(\bfx) = F(\bff)$ because $\bff$ is uniformly recurrent.
Therefore the identity $\calL_{\bfx} = \calL_{\bff}$ holds. 
Let us prove the converse, and suppose that $\calL_{\bfx} = \calL_{\bff}$ for some Fibonacci infinite word $\bff$. Then in particular, $\calL_{\bfx}(1) = \calL_{\bff}(1)=2$, so that $\bfx$ consists of two distinct letters, say $\wa$ and $\wb$ with $\wa < \wb$.
Since $\calL_{\bfx}(3) - \calL_{\bfx}(2) = 1$, exactly one of $\wa\wa\wb$ and $\wa\wb\wb$ occurs in~$\bfx$. If $\wa\wa\wb$ is in $L(\bfx)$,
we may assume that $\bff = \lim_{\ntoinf}f_{n}$ is the Fibonacci word with $f_{1} = \wb$ and $f_{2} = \wa$,
so that $\wa\wa\wb \in L(\bff)$. Similarly, if $\wa\wb\wb$ is in $L(\bfx)$, we may assume that $\bff = \lim_{\ntoinf}f_{n}$ is the Fibonacci word with $f_{1} = \wa$ and $f_{2} = \wb$,
so that $\wa\wb\wb \in L(\bff)$. Then a Lyndon word of length at most 3 is a factor of $\bfx$ if and only if it is a factor of~$\bff$.
We will show next that $L(\bfx) = L(\bff)$; then Theorem~\ref{060720122059} implies that $F(\bfx) = F(\bff)$
because both $\bfx$ and $\bff$  are recurrent.

If $L(\bfx) \neq L(\bff)$, then the identity $\calL_{\bfx} = \calL_{\bff}$ implies that
there exist an integer $k$ and distinct Lyndon words $w,z$ with $\nabs{w} = \nabs{z} = F_{k}$
such that $w$ is a factor of $\bfx$ and $z$ is a factor of~$\bff$. Let us assume that $k$ is as small as possible. 
Since the Lyndon factors of length at most 3 in $\bfx$ and $\bff$ coincide, we have $\nabs{w} > 3$, and thus $k\geq 5$.
Recall that $w$ can be written as $w = \lambda_{w}\mu_{w}$ where $\lambda_{w}$ and $\mu_{w}$ are Lyndon words by Lemma~\ref{pe 15-06-2012:0726}.
Furthermore, each of $\nabs{\lambda_{w}}$ and $\nabs{\mu_{w}}$ is a Fibonacci number, and therefore $\nabs{\lambda_{w}} + \nabs{\mu_{w}} = \nabs{w} = F_{k}$
yields $\nsset{\nabs{\lambda_{w}}, \nabs{\mu_{w}}} = \nsset{F_{k-1}, F_{k-2}}$.
The same reasoning shows that $z = \lambda_{z} \mu_{z}$ and $\nsset{\nabs{\lambda_{z}}, \nabs{\mu_{z}}} = \nsset{F_{k-1}, F_{k-2}}$.
But since $k\geq5$, we have $F_{k-2}\geq2$, which means that $\bsset{\lambda_{w}, \mu_{w}} = \bsset{\lambda_{z}, \mu_{z}}$ because the set of Lyndon factors of $\bfx$ of length less than $F_{k}$ coincides with the set of Lyndon factors of $\bff$ of length less than $F_{k}$ and there is precisely one Lyndon factor of length $F_{k-2}$ and precisely one Lyndon factor of length $F_{k-1}$.
Consequently, $z$ is a product of $\lambda_{w}$ and $\mu_{w}$. However, if $z = \lambda_{w}\mu_{w}$, then $z = w$, and if $z=\mu_{w}\lambda_{w}$, then $z$ is not a Lyndon word because it is a proper conjugate of the Lyndon word $\lambda_{w}\mu_{w} = w$, both of which are contradictions.
\end{proof}

\begin{remark}
Theorem~\ref{060720122147} shows that the shift orbit closure of a Fibonacci infinite word $\bff$ is characterized by the mapping~$\calL_{\bff}$, up to renaming letters.
But in general, the mapping $\calL_{\bfx}$ does not characterize a recurrent word~$\bfx$.
For example, the identity $\calL_{\bfx} = \calL_{\bfy}$ holds for the two periodic words $\bfx = (\word{000001})^{\omega}$ and $\bfy = (\word{000101})^{\omega}$, but clearly 
 $F(\bfx) \neq F(\bfy)$ and  $F(\bfx) \neq F(c(\bfy))$,  where $c$ is the morphism $\o \mapsto \i$, $\i \mapsto \o$.
\end{remark}

\section{Acknowledgement}

I thank Tero Harju for several useful discussions and for bringing the result of Siromoney et al.\ to my attention.


\begin{thebibliography}{99}
\bibitem{AllSha2003} J.--P. Allouche and J. Shallit. \emph{Automatic Sequences: Theory, Applications, Generalizations}. Cambridge University Press, 2003.
\bibitem{BerdeL1997} J. Berstel and A. de Luca. Sturmian words, Lyndon words and trees. \emph{Theoret. Comput. Sci.}~\textbf{178} (1997), 171--203.
\bibitem{Cassaigne2008} J. Cassaigne. On extremal properties of the Fibonacci word. \emph{RAIRO-Theor. Inf. Appl.}~\textbf{42}(4) (2008), 701--715.
\bibitem{CasMigRes1999} M. G. Castelli, F. Mignosi, and A. Restivo. Fine and Wilf's theorem for three periods and a generalization of Sturmian words.
\emph{Theoret. Comput. Sci.}~\textbf{218} (1999), 83--94.
\bibitem{ChaHanPer2004} J.-M. Champarnaud, G. Hansel, and D. Perrin. Unavoidable sets of constant length. \emph{International Journal of Algebra and Computation}~\textbf{14}(2) (2004), 241--251.
\bibitem{CheFoxLyn1958} K. T. Chen, R. H. Fox, and R. C. Lyndon. Free differential calculus, IV. The quotient groups of the lower central series.
\emph{Ann. Math.}, 2nd Ser.,~\textbf{68}(1) (1958), 81--95.
\bibitem{CroRyt1995} M. Crochemore and W. Rytter. Squares, cubes, and time-space efficient string searching. \emph{Algorithmica}~\textbf{13} (1995), 405--425.
\bibitem{CurSaa2009} J. D. Currie and K. Saari. Least periods of factors of infinite words. \emph{RAIRO-Theor. Inf. Appl.}~\textbf{43} (2009), 165--178.
\bibitem{GocSaaSha2012} D. Go\v{c}, K. Saari, and J. Shallit. Primitive words and Lyndon words in automatic and linearly recurrent sequences. Submitted.
\bibitem{deLuca1981} A. de Luca. A combinatorial property of the Fibonacci words. \emph{Inf. Process. Letters}~\textbf{12}(4) (1981), 193--195.
\bibitem{DieHarNow2010} V. Diekert, T. Harju, and D. Nowotka. Weinbaum factorizations of primitive words. \emph{Russian Mathematics (Iz VUZ)}~\textbf{54}(1) (2010), 16--25.
\bibitem{Duval1983} J.--P. Duval. Factorizing words over an ordered alphabet. \emph{J. Algorithms}~\textbf{4} (1983), 363--381.
\bibitem{EhrSil1979} A. Ehrenfeucht and 	D.~M.~Silberger. Periodicity and unbordered segments of words. \emph{Discrete Math.}~\textbf{26} (1979), 101--109.
\bibitem{Fischler2006} S. Fischler. Palindromic prefixes and episturmian words. \emph{J. Combin. Theory Ser. A} \textbf{113} (2006), 1281--1304.
\bibitem{FinWil1965} N. J. Fine and H. S. Wilf. Uniqueness theorems for periodic functions. \emph{Proc. Amer. Math. Soc.}~\textbf{16}(1) (1965), 109--114.
\bibitem{FreMai1978} H. Fredricksen and J. Maiorana. Necklaces of beads in $k$ colors and $k$-ary de Bruijn sequences. \emph{Discrete Math.}~\textbf{23} (1978), 207--210.

\bibitem{GatFlo2012} T. Gateva--Ivanova and G. Fl\obackslash{}ystad. Monomial algebras defined by Lyndon words. Preprint available at \url{http://arxiv.org/abs/1207.6256}.
\bibitem{Knuth} D. E. Knuth. \emph{The Art of Computer Programming}. Addison--Wesley, Reading, Mass., 1968.
\bibitem{Lothaire1997} M. Lothaire. \emph{Combinatorics on Words},  Cambridge University Press, Cambridge, 1997.
\bibitem{Lothaire2002} M. Lothaire. \emph{Algebraic Combinatorics on Words},  Cambridge University Press, Cambridge, 2002.
\bibitem{MigResSal1998} F. Mignosi, A. Restivo, and S. Salemi. Periodicity and the golden ratio. \emph{Theoret. Comput. Sci.}~\textbf{204} (1998), 153--167.
\bibitem{SirMatDarSub1994} R. Siromoney, L. Mathew, V.~R.~Dare, and K.~G.~Subramanian. Infinite Lyndon Words. \emph{Inf. Process. Letters}~\textbf{50} (1994), 101--104.
\bibitem{TijZam2009} R. Tijdeman and L. Zamboni. Fine and Wilf words for any periods II. \emph{Theoret. Comput. Sci.}~\textbf{410} (2009), 3027--3034.
\bibitem{WenWen1994} Z.-X. Wen and Z.-Y. Wen. Some properties of the singular words of the Fibonacci word. \emph{Europ. J. Combinatorics}~\textbf{15} (1994), 587--598.
\end{thebibliography}
\end{document}